\documentclass[12pt,a4paper]{amsart}
\usepackage{amsmath,amsfonts,amssymb,amscd,amsthm}
\usepackage[shortlabels]{enumitem}

\usepackage{color, soul}
\definecolor{gr}{rgb}{0.7, 1, 0.7}
\definecolor{rr}{rgb}{1, 0.7, 0.7}

\usepackage{latexsym}
\usepackage{graphicx}

\theoremstyle{plain} 
\newtheorem{theorem}{Theorem}[section]
\newtheorem{lemma}[theorem]{Lemma}
\newtheorem{corollary}[theorem]{Corollary}

\newtheorem{proposition}[theorem]{Proposition}

\theoremstyle{definition} 
\newtheorem{definition}[theorem]{Definition}
\theoremstyle{remark} 
\newtheorem{remark}[theorem]{Remark}

\renewcommand{\mathfrak}{\mathbf}

\newcommand{\ignore}[1]{}

\newcommand{\eps}{\varepsilon}

\newcommand{\bbC}{\mathbb{C}}

\newcommand{\bbN}{\mathbb{N}}
\newcommand{\bbR}{\mathbb{R}}
\newcommand{\bbZ}{\mathbb{Z}}

\newcommand{\NN}{\mathbb N}

\newcommand{\RR}{\mathbb R}

\newcommand{\cR}{\mathcal R}

\newcommand{\dist}{\operatorname{dist}}
\newcommand{\diam}{\operatorname{diam}}

\renewcommand{\mod}{\operatorname{mod}}

\begin{document}
\title{Renormalization for unimodal maps with non-integer exponents}
\thanks{MY was partially supported by NSERC Discovery Grant}
\begin{abstract}
We define an analytic setting for renormalization of unimodal maps with an arbitrary critical exponent. We prove the global Hyperbolicity of Renormalization conjecture for unimodal maps of bounded type with a critical exponent which is sufficiently close to an even integer.
\end{abstract}

\author{Igors Gorbovickis}
\address{Uppsala University, Uppsala, Sweden}
\email{igors.gorbovickis@math.uu.se}

\author{Michael Yampolsky}
\address{University of Toronto, Toronto, Canada}
\email{yampol@math.toronto.edu}

\subjclass[2010]{}

\date{\today}
\maketitle

\section{Preliminaries and statements of results}
Renormalization theory of unimodal maps has been the cornerstone of the modern development of one-dimensional real and complex dynamics. Seminal works of Sullivan \cite{Sull-qc,deMelo_vanStrien} and Douady and Hubbard \cite{DH85} put Feigenbaum-Collet-Tresser (FCT) universality conjectures into the context of holomorphic dynamics. Renormalization theory of analytic unimodal maps of the interval was completed by McMullen \cite{McM-ren2} and Lyubich \cite{Lyubich-Hairiness}. 

However, it was known since the early days of the development of the renormalization theory, that FCT-type universality was also observed in families of smooth unimodal maps of the interval with a critical point of an {\it arbitrary} order $\alpha>1$, such as, for instance, the unimodal family
$$x\mapsto |x|^\alpha+c,\; c\in\RR.$$
Of course, if $\alpha$ is not an even integer, such unimodal maps do not have single-valued analytic extensions to the complex plane, and therefore, the existing theory does not cover these cases.

One of the historical directions in which the problem of non-even integer exponents has been attacked are attempts to develop a purely "real" renormalization theory; that is, one that does not rely on complex-analytic techniques. Let us mention in this regard the beautiful paper of Martens \cite{Martens1} in which periodic orbits of unimodal renormalization are constructed for an arbitrary $\alpha$; as well as the work of Cruz and Smania \cite{SmCr} which continues Martens' approach.

We take a different approach in our paper. Below, we will construct a suitable {\it analytic} setting for renormalization of unimodal maps with  $\alpha\notin 2\NN$. In this setting, the renormalization operators become a smoothly parametrized family $\cR_\alpha$. Because of this, we are able to continue the renormalization hyperbolicity results of Sullivan, McMullen, and Lyubich to the values of exponents which are sufficiently close to even integers. For such $\alpha$'s, we obtain a {\it global} FCT renormalization horseshoe picture for maps of bounded combinatorial type, completely settling the theory in these cases.

Although our main results are perturbative, the construction of the family $\cR_\alpha$ works for all $\alpha>1$, thus giving us a general setting in which the theory may potentially be completed some day.

Let us note that this work grew out of our previous work on renormalization of non-analytic critical circle maps \cite{GorYam}. There is a  rather significant difference: in the setting of analytic critical circle maps, the definition of renormalization and Banach spaces it acts upon is rather involved (see \cite{Ya3}) and includes highly non-trivial changes of coordinates. The unimodal setting is, in comparison, very straightforward, making the proofs much shorter and technically less involved.

Our main results are Theorem~\ref{Global_attractor_theorem_expanded} in which we construct the global hyperbolic renormalization horseshoe for even unimodal maps ($f(x)=f(-x)$) of bounded type, and 
Theorem~\ref{non_even_theorem} in which we extend the result to non-even maps. We state these theorems in the following section, after having made the appropriate definitions. 
As a direct application of our main results, in Corollary~\ref{rigidity_corollary} we prove $C^{1+\beta}$-rigidity of Cantor attractors for infinitely renormalizable maps with  bounded combinatorics.


\subsection{Unimodal maps and their renormalization}

\begin{definition}
Let $\alpha$ be a real number such that $\alpha>1$. A smooth map $f\colon[-1,1]\to[-1,1]$ is unimodal with critical exponent $\alpha$, if there exists a point $c=c_f\in(-1,1)$, such that 
\begin{enumerate}[(i)]
 \item $f'(x)>0$, for all $x\in[-1,c)$;
 \item $f'(x)<0$, for all $x\in(c,1]$;
 \item in a neighborhood of the point $c$, the function $f$ can be reperesented as 
\begin{equation}\label{f_equals_psi_alpha_phi_equation}
f(x)=\psi(-|\phi(x)|^\alpha),
\end{equation}
where $\phi(c)=0$, and $\phi$ and $\psi$ are local orientation preserving diffeomorphisms in some neighborhoods of $c$ and $0$ respectively.
\end{enumerate}
\end{definition}

We denote the space of all unimodal maps by $\mathfrak U$.
We will say that a unimodal map $f$ is $C^k$-smooth ($C^\infty$-smooth, or analytic), if $f$ is of class $C^k$ ($C^\infty$, or analytic) on the intervals $[-1,c_f)$ and $(c_f,1]$, and there exists a decomposition~(\ref{f_equals_psi_alpha_phi_equation}), such that $\phi$ and $\psi$ are also of class $C^k$ ($C^\infty$, or analytic) in some neighborhoods of $c_f$ and $0$ respectively. 
The space of all analytic unimodal maps will be denoted by $\mathfrak U^\omega$.

\begin{definition}\label{renormalizable_def}
A unimodal map $f$ is \textit{renormalizable}, if there exists an integer $m\ge 2$ and a closed interval $J=[a,b]\subset\bbR$, such that $c_f\in(a,b)$, $f^m(J)\subset J$ and the intervals $J, f(J),\dots, f^{m-1}(J)$ have pairwise disjoint interior. The smallest $m$ with this property is called the \textit{renormalization period} of $f$.
\end{definition}

For two points $p,q\in\bbC$, $p\neq q$, let $A_{p,q}\colon\bbC\to\bbC$ denote the linear map
$$A_{p,q}(z)=\frac{(z-p)+(z-q)}{q-p},$$ so that $A_{p,q}(q)=1$ and $A_{p,q}(p)=-1$. 

\begin{definition}\label{renormalization_def}
Assume that a unimodal map $f\colon[-1,1]\to[-1,1]$ is renormalizable with period $m$, and let $J=[a,b]$ be the maximal interval satisfying the conditions of Definition~\ref{renormalizable_def}. 
Let $p=a$, $q=b$, if $(f^m)'(a)>0$ and $p=b$, $q=a$ otherwise. Then the map
$$
\cR(f)=A_{p,q}\circ f^m\circ A_{p,q}^{-1}
$$
is called the \textit{renormalization} of $f$.
\end{definition}

It is easy to check that the map $\cR(f)$ is also unimodal with the same critical exponent. 
If $\cR(f)$ is also renormalizable, then we say that $f$ is \textit{twice renormalizable}. This way we define $n$ times renormalizable unimodal maps, for all $n=1, 2, 3,\dots$, including $n=\infty$.

If a unimodal map $f$ is renormalizable with period $m$, then the relative order of the intervals $J, f(J),f^2(J),\dots,f^{m-1}(J)$ inside $[-1,1]$ determines a permutation $\theta(f)$ of $\{0,1,\dots,m-1\}$. A permutation $\theta$ is called \textit{unimodal}, if there exists a renormalizable unimodal map $f$, such that $\theta=\theta(f)$. The set of all unimodal permutations will be denoted by $\mathbf P$. 

For $n\in\bbN\cup\{\infty\}$ and a subset $\Theta\subset\mathbf P$, let $\mathcal S_\Theta^n$ be the set of all $n$ times renormalizable unimodal maps $f$, such that $\theta(\cR^j(f))\in\Theta$, for all $j=0,1,2,\dots,n-1$. For $f\in\mathcal S_{\mathbf P}^\infty$, let $\rho(f)$ be the infinite sequence of permutations $(\theta(f),\theta(\cR(f)),\theta(\cR^2(f)),\dots)\subset\mathbf P^\bbN$. 

We say that two infinitely renormalizable unimodal maps $f$ and $g$ are of the same combinatorial type, if $\rho(f)=\rho(g)$.


\subsection{Hyperbolic renormalization attractor}
In the remaining part of the paper we will work only with analytic unimodal maps from $\mathfrak U^\omega$.



For a compact set $K\subset\bbC$ and a positive real number $r>0$, let $N_r(S)$ denote the $r$-neighborhood of $K$ in $\bbC$, namely,
$$
N_r(K)=\{z\in\bbC\mid \min_{w\in K}|z-w|<r\}.
$$




For a Jordan domain $\Omega\subset\bbC$, let $\mathcal B(\Omega)$ denote the space of all analytic maps $f\colon\Omega\to\bbC$ that continuously extend to the closure $\overline\Omega$. The set $\mathcal B(\Omega)$ equipped with the sup-norm, is a complex Banach space.
If $\Omega$ is symmetric with respect to the real axis, we let $\mathcal B^\bbR(\Omega)\subset\mathcal B(\Omega)$ denote the real Banach space of all real-symmetric functions from $\mathcal B(\Omega)$.


\begin{definition}
For a positive real number $r>0$, let $\tilde{\mathfrak B}_r\subset\mathcal B^\bbR(N_r([-1,0]))$ be the set of all maps $\psi\in\mathcal B^\bbR(N_r([-1,0]))$ that are univalent in some neighborhood of the interval $[-1,0]$, and such that $\psi(-1)=-1$, $-1<\psi(0)\le 1$. Let $\mathfrak B_r\subset\tilde{\mathfrak B}_r$ be the subset of all $\psi\in\tilde{\mathfrak B}_r$, such that $\psi(0)<1$. 
\end{definition}

\begin{proposition}\label{B_r_alpha_submanifold_prop}
For any real number $r>0$, the sets $\mathfrak B_r$ and $\tilde{\mathfrak B}_r\setminus \mathfrak B_r$ are respectively codimension~$1$ and codimension~$2$ affine submanifolds of $\mathcal B^\bbR(N_r([-1,0]))$.
\end{proposition}
\begin{proof}
Let $\mathcal B^\bbR_{r,-1}$ 
denote the Banach subspace of $\mathcal B^\bbR(N_r([-1,0]))$ that consists of all $\psi\in\mathcal B^\bbR(N_r([-1,0]))$, such that $\psi(-1)=0$.
Then, $\mathfrak B_r$ is an open subset of the affine Banach space $-1+\mathcal B^\bbR_{r,-1}$. Similarly, $\tilde{\mathfrak B}_r\setminus \mathfrak B_r$ is a codimension~$1$ affine submanifold of $-1+\mathcal B^\bbR_{r,-1}$.
\end{proof}

For each positive $\alpha>1$ we define a map $j_\alpha\colon\tilde{\mathfrak B}_r\to\mathfrak U$ that associates a unimodal map to every element of $\tilde{\mathfrak B}_r$ according to the formula
$$
[j_\alpha(\psi)](x)= \psi(-|x|^\alpha).
$$
Clearly, for every $\alpha>1$, the map $j_\alpha$ is one-to-one.

\begin{definition}
For real numbers $r>0$, $\alpha>1$, let $\tilde{\mathfrak A}_r^\alpha,\mathfrak A_r^\alpha\subset\mathfrak U$ be the spaces of analytic unimodal maps with critical exponent $\alpha$, defined by
$$
\tilde{\mathfrak A}_r^\alpha=j_\alpha(\tilde{\mathfrak B}_r),\qquad\text{and}\qquad\mathfrak A_r^\alpha=j_\alpha(\mathfrak B_r).
$$
Let $\mathfrak A^\alpha$ and $\tilde{\mathfrak A}^\alpha$ denote the spaces of all unimodal maps $f$, such that $f\in\mathfrak A_r^\alpha$ and $f\in\tilde{\mathfrak A}_r^\alpha$ respectively, for some $r>0$. 
\end{definition}

For each $\alpha>1$, the space $\mathfrak A_r^\alpha$ has a structure of a real affine Banach manifold 
inherited from $\mathfrak B_r$. The Banach manifold structure induces a metric $\dist_r(\cdot,\cdot)$ on $\tilde{\mathfrak A}_r^\alpha$, defined as follows: for any pair of maps $f_1,f_2\in\tilde{\mathfrak A}_r^\alpha$, such that $f_1=j_\alpha(\psi_1)$ and $f_2=j_\alpha(\psi_2)$, 
\begin{equation}\label{dist_r_metric_eq}
\dist_r(f_1,f_2)=\sup_{z\in N_r([-1,0])}|\psi_1(z)-\psi_2(z)|.
\end{equation}

Our main result is the following theorem, which extends the Sullivan-McMullen-Lyubich FCT hyperbolicity of renormalization to unimodal maps with critical exponents $\alpha$ close to even integers:

\begin{theorem}[{\bf Hyperbolic renormalization attractor}]\label{Global_attractor_theorem_expanded}
For every $k\in\bbN$ and a non-empty finite set $\Theta\subset\mathbf P$, there exist an open interval $J=J(k,\Theta)\subset\bbR$ containing the number $2k$, a positive real number $r=r(k)>0$, and a positive integer $N=N(k)\in\bbN$, such that 
for every $\alpha\in J$, there exist
an open set $\mathcal O_\alpha=\mathcal O_\alpha(\Theta)\subset\mathfrak A_r^\alpha\cap\mathcal S_{\Theta}^N$ 
and an $\cR$-invariant compact set $\mathcal I_\Theta^\alpha\subset\mathcal O_\alpha\cap S_\Theta^\infty$ with the following properties.
\begin{enumerate}[(i)]
\item (horseshoe property): The action of $\cR$ on $\mathcal I_\Theta^\alpha$ is topologically conjugate to the two-sided shift $\sigma\colon\Theta^\bbZ\to\Theta^\bbZ$:
$$
\iota_\alpha\circ\mathcal R\circ \iota_\alpha^{-1}=\sigma,
$$
and if 
$$
f=\iota^{-1}_\alpha(\dots,\theta_{-k},\dots,\theta_{-1},\theta_0,\theta_1,\dots,\theta_k,\dots),
$$
then 
$$
\rho(f)=[\theta_0,\theta_1,\dots,\theta_k,\dots].
$$

\item (global stable sets): For every $f\in\mathbf A^\alpha\cap S_\Theta^\infty$, there exists $M\in\bbN$, such that for all $m\ge M$ the renormalizations $\cR^m(f)$ belong to $\mathbf A_r^\alpha$ and for every $g\in\mathcal I_\Theta^\alpha$ with $\rho(f)=\rho(g)$, we have
\begin{equation}\label{pair_exp_convergence_eq}
\dist_r(\cR^m(f),\cR^m(g))\le C\lambda^m,
\end{equation}
for some constants $C>0$, $\lambda\in(0,1)$ that depend only on $\Theta$ and~$\alpha$.

\item (hyperbolicity): $\cR^N(\mathcal O_\alpha)\subset\mathfrak A_r^\alpha$, the operator $\cR^N\colon \mathcal O_\alpha\to\mathfrak A_r^\alpha$ is analytic, and $\mathcal I_\Theta^\alpha$ is a locally maximal uniformly hyperbolic set for $\cR^N$ with a one-dimensional unstable direction.
\end{enumerate}
\end{theorem}

We note that for an even renormalizable unimodal map $f\in\mathfrak U^\omega$, we have $\cR(f)\in\mathbf A^\alpha$, hence 
Theorem~\ref{Global_attractor_theorem_expanded} settles the Renormalization Hyperbolicity Conjecture for the space of all even unimodal maps from $\mathfrak U^\omega$ with bounded combinatorial type and critical exponents sufficiently close to $2\bbN$. In the following general theorem we extend the results of Theorem~\ref{Global_attractor_theorem_expanded} to the case of general (i.e. not necessarily even) analytic unimodal maps. In order to state the theorem, we start with some definitions.

\begin{definition}
For a positive real number $r>0$, let $\mathbf{\Phi}_r\subset\mathcal B^\bbR(N_r([-1,1]))$ be the set of all maps $\phi\in\mathcal B^\bbR(N_r([-1,1]))$ that are univalent in some neighborhood of the interval $[-1,1]$, and such that $\phi(-1)=-1$, and $\phi(1)=1$.
\end{definition} 

\begin{proposition}\label{Phi_r_submanifold_prop}
For any real number $r>0$, the set $\mathbf{\Phi}_r$ is a codimension~$2$ affine submanifolds of $\mathcal B^\bbR(N_r([-1,1]))$.
\end{proposition}
\begin{proof}
The proof is analogous to the proof of Proposition~\ref{B_r_alpha_submanifold_prop}.
\end{proof}

For $\phi\in\mathbf{\Phi}_r$, let ${\mathbf A}^\alpha(\phi)\subset\mathfrak U^\omega$ be the set of all $g\in\mathfrak U^\omega$, such that 
\begin{equation}\label{f_conjug_eq}
g=\phi^{-1}\circ f\circ\phi,
\end{equation}
for some $f\in{\mathbf A}^\alpha$. 
Let ${\mathbf A}^\alpha(\mathbf{\Phi}_r)\subset\mathfrak U^\omega$ be the union
$$
{\mathbf A}^\alpha(\mathbf{\Phi}_r)=\bigcup_{\phi\in\mathbf{\Phi}_r} {\mathbf A}^\alpha(\phi).
$$
It is easy to check that if $f\in{\mathbf A}^\alpha$ is renormalizable with the affine rescaling $A_{p,q}$ as in Definition~\ref{renormalization_def}, 
then the map $g$ from~(\ref{f_conjug_eq}) is also renormalizable and
$$
\cR(g)= [\mathcal F_f(\phi)]^{-1}\circ\cR(f)\circ\mathcal F_f(\phi),
$$
where
$$
\mathcal F_f(\phi)= A_{p,q}\circ\phi\circ A_{\phi^{-1}(p),\phi^{-1}(q)}^{-1} \in\mathbf{\Phi}_r.
$$
This allows us to define the operator $\tilde{\cR}\colon(\mathcal S_{\mathbf P}^1\cap\mathbf A^\alpha)\times\mathbf{\Phi}_r\to\tilde{\mathbf A}^\alpha\times\mathbf{\Phi}_r$ as a skew product
$$
\tilde{\cR}(f,\phi)=(\cR(f),\mathcal F_f(\phi)).
$$
For $f\in\mathcal S_{\mathbf P}^\infty\cap\mathbf A^\alpha$ and $n\in\bbN$, let $\mathcal F_f^n(\phi)$ denote the map $\phi_n\in\mathbf{\Phi}_r$, such that $\tilde{\cR}^n(f,\phi)=(\cR^n(f),\phi_n)$. Let $\|\cdot\|_r$ denote the Banach norm in $\mathcal B(N_r([-1,1]))$.

The following theorem reduces the general case of analytic unimodal maps to the case of even ones. The proof follows from real a priori bounds (c.f.~\cite{deMelo_vanStrien}) and will be given in Section~\ref{non_even_sec}.

\begin{theorem}\label{non_even_theorem}
(i) For every $g\in\mathfrak U^\omega\cap\mathcal S_{\mathbf P}^\infty$ with critical exponent $\alpha$ and for a positive real number $r>0$, there exists $K_1=K_1(g,r)\in\bbN$, such that for every $k\ge K_1$, we have $\cR^k(g)\in\mathbf A^\alpha(\mathbf{\Phi}_r)$.

(ii) For every pair of real numbers $r>0$, $\alpha\in(1,+\infty)$ and for every $f\in\mathbf A^\alpha\cap\mathcal S_{\mathbf P}^\infty$ and $\phi\in\mathbf{\Phi}_r$, there exists $K_2=K_2(f, \phi)\in\bbN$, such that for all $k\ge K_2$, we have
$$
\|\mathcal F_f^k(\phi)-\mathrm{id}\|_r\le C\lambda^k,
$$
for some constants $C>0$, $\lambda\in(0,1)$ that depend only on $\alpha$.
\end{theorem}

As an immediate corollary of Theorem~\ref{Global_attractor_theorem_expanded} and Theorem~\ref{non_even_theorem}, we state the following rigidity result:
\begin{corollary}[{\bf $C^{1+\beta}$-rigidity}]\label{rigidity_corollary}
Let $\Theta\subset \mathbf P$ be a non-empty finite set. Then for every pair of maps $f, g\in\mathfrak U^\omega\cap\mathcal S_\Theta^\infty$ with $\rho(f)=\rho(g)$ and with the same critical exponent $\alpha\in\cup_{k\in\bbN}J(k,\Theta)$, 
there exists a $C^{1+\beta}$ diffeomorphism $h\colon\bbR\to\bbR$, 
that conjugates $f$ and $g$ on their corresponding attracting Cantor sets. The constant $\beta>0$ depends only on $\alpha$ and $\Theta$.
\end{corollary}
\begin{proof}
The corollary follows directly from Theorem~\ref{Global_attractor_theorem_expanded} and Theorem~\ref{non_even_theorem} together with Theorem~9.4 from Chapter~VI of~\cite{deMelo_vanStrien}.
\end{proof}





For ease of reference, let us quote a theorem from \cite{DeFaria_DeMelo_Pinto}
who state the 
Sullivan-McMullen-Lyubich renormalization hyperbolicity theorem for the case when the critical exponent is an even integer in a convenient for us form:

\begin{theorem}\label{renormalization_2k_theorem}
For every $k\in\bbN$ and a non-empty finite set $\Theta\subset\mathbf P$, there exist a positive real number $r=r(k)>0$, a positive integer $N=N(k)\in\bbN$, an open set $\mathcal O_{2k}=\mathcal O_{2k}(\Theta)\subset\mathfrak A_r^{2k}\cap\mathcal S_{\Theta}^N$ and an $\cR$-invariant compact set $\mathcal I_\Theta^{2k}\subset\mathcal O_{2k}\cap S_\Theta^\infty$, such that all properties from Theorem~\ref{Global_attractor_theorem_expanded} hold for $\alpha=2k$.
Furthermore, all maps from the image $\cR^N(\mathcal O_{2k})$ belong to $\mathfrak A_{2r}^{2k}$.
\end{theorem}

\begin{remark}\label{r_small_remark}
It follows from the proof of Theorem~\ref{renormalization_2k_theorem}, provided in~\cite{DeFaria_DeMelo_Pinto}, that the positive real number $r>0$ can be chosen arbitrarily small. In this case the positive integer $N$ and the set $\mathcal O_{2k}$ depend on $r$.
\end{remark}

\section{Proof of Theorem~\ref{Global_attractor_theorem_expanded}}

In this section we give a proof of Theorem~\ref{Global_attractor_theorem_expanded}. The proof is split into two lemmas. Roughly speaking, the first lemma proves property~(iii), and the second lemma proves properties (ii) and (i) of Theorem~\ref{Global_attractor_theorem_expanded}. The properties are proved precisely in the reverse order: (iii) $\implies$ (ii) $\implies$ (i). Before we formulate these lemmas, let us start with a definition:

\begin{definition}
For a positive real number $r>0$ and a set $I\subset(1,+\infty)$, let $\tilde{\mathfrak A}_r^I$ and $\mathfrak A_r^I$ be the disjoint unions $\tilde{\mathfrak A}_r^I\equiv\coprod_{\alpha\in I}\tilde{\mathfrak A}_r^\alpha$ and $\mathfrak A_r^I\equiv\coprod_{\alpha\in I}\mathfrak A_r^\alpha$.
Let $\tilde{\mathfrak A}^I$ and $\mathfrak A^I$ denote the spaces of all unimodal maps $f$, such that $f\in\tilde{\mathfrak A}_r^I$ and $f\in\mathfrak A_r^I$ respectively, for some $r>0$.
\end{definition}

If $I$ is an open set, then $\mathfrak A_r^I$ is a Banach manifold, diffeomorphic to $\mathfrak B_r\times I$. We extend the metric $\dist_r$ to $\tilde{\mathfrak A}_r^I$ in the following way: if $f_1, f_2\in\tilde{\mathfrak A}_r^I$ are two unimodal maps with critical exponents $\alpha_1$ and $\alpha_2$ respectively, such that $f_1=j_{\alpha_1}(\psi_1)$ and $f_2=j_{\alpha_2}(\psi_2)$, then
$$
\dist_r(f_1,f_2)=|\alpha_1-\alpha_2|+\sup_{z\in N_r([-1,0])}|\psi_1(z)-\psi_2(z)|.
$$

\begin{lemma}[Property (iii) of Theorem~\ref{Global_attractor_theorem_expanded}]\label{hyperb_lemma}
For every $k\in\bbN$ and a non-empty finite set $\Theta\subset\mathbf P$, there exist an open interval $J_1=J_1(k,\Theta)\subset\bbR$ containing the number $2k$, a positive real number $r=r(k)>0$ and a positive integer $N=N(k)\in\bbN$, such that for every $\alpha\in J_1$, 
there exist
an open set $\mathcal O_\alpha=\mathcal O_\alpha(\Theta)\subset\mathfrak A_r^\alpha\cap\mathcal S_{\Theta}^N$ 
and an $\cR^N$-invariant compact set $\mathcal I_\Theta^\alpha\subset\mathcal O_\alpha\cap S_\Theta^\infty$ that satisfies property (iii) of Theorem~\ref{Global_attractor_theorem_expanded}. The action of $\cR^N$ on ${\mathcal I}_\Theta^{\alpha}$ is topologically conjugate to the action of $\cR^N$ on $\mathcal I_\Theta^{2k}$ by a homeomorphism $h_\alpha\colon\mathcal I_\Theta^{2k}\to\mathcal I_\Theta^{\alpha}\subset\mathfrak A_r^{J_1}$ that continuously depends on $\alpha\in J_1$, and $h_{2k}=id$. 
\end{lemma}

\begin{lemma}[Properties (i) and (ii) of Theorem~\ref{Global_attractor_theorem_expanded}]\label{global_lemma}
For every $k\in\bbN$ and a non-empty finite set $\Theta\subset\mathbf P$, let $J_1$, $r$ and the sets $\mathcal I_\Theta^\alpha$, where $\alpha\in J_1$, be the same as in Lemma~\ref{hyperb_lemma}. Then there exists an open interval $J\subset J_1$ containing the number $2k$, such that for every $\alpha\in J$, properties~(i) and~(ii) of Theorem~\ref{Global_attractor_theorem_expanded} hold.
\end{lemma}


\subsection{Extending hyperbolicity}

First, we prove property (iii) of Theorem~\ref{Global_attractor_theorem_expanded}.



\begin{proof}[Proof of Lemma~\ref{hyperb_lemma}]
Fix $k\in\bbN$ and a finite non-empty set $\Theta\subset\mathfrak P$. Let the constants $r>0$ and $N\in\bbN$ as well as the sets $\mathcal O_{2k}$ and $\mathcal I_\Theta^{2k}$ be the same as in Theorem~\ref{renormalization_2k_theorem}.

Define the set $\hat{\mathcal I}_\Theta^{2k}=j_{2k}^{-1}(\mathcal I_\Theta^{2k})$. 
Let $I\subset\bbR$ be an open interval, such that $2k\in I$. Then from boundedness of combinatorics (finiteness of $\Theta$) and continuity arguments it follows that there exists an open set $\mathcal O\subset \mathfrak A_r^I\cap\mathcal S_\Theta^N$, such that $\mathcal O_{2k}\subset \mathcal O$ and $\cR^N(\mathcal O)\subset \mathfrak A_{3r/2}^I\subset \mathfrak A_{r}^I$. The operator $\cR^N\colon\mathcal O\to\mathfrak A_r^I$ is real-analytic, since it is a rescaling of a finite composition, and the rescaling depends analytically on the map. 

Let $\tilde I\subset I$ be an open interval, such that for any $\alpha\in \tilde I$, the set $\mathcal O_\alpha\equiv\mathcal O\cap\mathfrak A_r^\alpha$ is non-empty, and there exists an open set $\mathcal U\subset\mathfrak B_r$, such that for all $\alpha\in \tilde I$, the operators $\cR_\alpha\equiv j_\alpha^{-1}\circ\cR^N\circ j_\alpha$ are defined in $\mathcal U$, the image $\cR_\alpha(\mathcal U)$ is contained in $\mathfrak B_r$ and $\hat{\mathcal I}_\Theta^{2k}\subset\mathcal U$. Clearly, the operators $\cR_\alpha\colon\mathcal U\to\mathfrak B_r$ are real-analytic and analytically depend on $\alpha\in \tilde I$. 

It follows from Theorem~\ref{renormalization_2k_theorem} that the set $\hat{\mathcal I}_\Theta^{2k}\subset\mathfrak B_r$ is invariant and uniformly hyperbolic for the operator $\cR_{2k}$ with a one-dimensional unstable direction. Furthermore, the action of $\cR_{2k}$ on $\hat{\mathcal I}_\Theta^{2k}$ is topologically conjugate to the two-sided shift $\sigma^N$ on $\Theta^\bbZ$. Now it follows from the theorem on structural stability of hyperbolic sets that there exists an open interval $J_1\subset \tilde I$, such that $2k\in J_1$, and for every $\alpha\in J_1$, the operator $\cR_\alpha$ has an invariant uniformly hyperbolic set $\hat{\mathcal I}_\Theta^{\alpha}\subset\mathcal U$ with a one-dimensional unstable direction. Furthermore, the action of $\cR_{\alpha}$ on $\hat{\mathcal I}_\Theta^{\alpha}$ is topologically conjugate to the two-sided shift $\sigma^N$ on $\Theta^\bbZ$. Finally, for each $\alpha\in J_1$ we define $\mathcal I_\Theta^\alpha\equiv j_\alpha(\hat{\mathcal I}_\Theta^\alpha)$, which completes the proof.
\end{proof}


\subsection{Complex bounds}

For a set $S\subset\bbC$, by $-S\subset\bbC$ we denote the reflection of $S$ about the origin. In other words,
$$
-S=\{z\in\bbC\mid -z\in S\}.
$$
For $\alpha\in(1,+\infty)$, let $p_\alpha\colon\bbC\setminus(0,+\infty)\to\bbC$ be the branch of the map $z\mapsto -(-z)^\alpha$, such that $p_\alpha((-\infty,0])=(-\infty,0]$.

\begin{definition}
For a simply connected domain $U\subset\bbC$ and a set $X\Subset U$, let
$
\mod(X,U)
$
denote the supremum of the moduli of all annuli $A\subset U\setminus\overline X$ that separate $\partial U$ from $\overline X$.
\end{definition}

\begin{definition}\label{H_mu_def}
For a set $I\subset(1,+\infty)$ and a real number $\mu\in(0,1)$, let $\mathbf H^I(\mu)\subset\tilde{\mathfrak A}^I$ be the set of all unimodal maps $f\in\tilde{\mathfrak A}^I$ of the form $f(x)=\psi(-|x|^\alpha)$, where $\alpha\in I$, and $\psi$ is a univalent analytic map of some simply connected neighborhood $U_f\subset\bbC$ of the interval $[-1,0]$, such that
\begin{enumerate}[(i)]
\item $\diam (\psi(U_f))< 1/\mu$;
\item the neighborhood $V_f=p_\alpha^{-1}(U_f)\cup -p_\alpha^{-1}(U_f)$ is compactly contained in $\psi(U_f)$, and $\mod(V_f,\psi(U_f))>\mu$.

\end{enumerate}






\end{definition}

\begin{lemma}\label{definite_r_lemma}
For a real number $\mu\in(0,1)$, there exists a positive real number $s=s(\mu)>0$, such that for every $I\subset(1,+\infty)$ and every $f\in\mathbf H^I(\mu)$ with critical exponent $\alpha$, the map $\psi=j_\alpha^{-1}(f)$ belongs to $\tilde{\mathfrak B}_s$ and is defined and univalent in $N_{2s}([-1,0])$. Furthermore, the inclusion 
\begin{equation}\label{psi_Uf_inclusion_eq}
N_s([-1,1])\subset\psi(U_f)
\end{equation}
holds.
\end{lemma}
\begin{proof}
Since $p_\alpha([-1,0])=[-1,0]$, the neighborhood $p_\alpha^{-1}(U_f)$ contains the interval $[-1,0]$, which implies that $[-1,1]\subset V_f$. 
According to the definition of the space $\tilde{\mathfrak B}_s$, we have $\psi([-1,0])\subset[-1,1]$, so $[-1,0]\subset\psi^{-1}(V_f)$. From this we conclude that
$$
\mod([-1,1], \psi(U_f))\ge \mod(V_f,\psi(U_f))>\mu, 
$$
and
$$
\mod([-1,0], U_f)\ge \mod(\psi^{-1}(V_f), U_f)=\mod(V_f,\psi(U_f))>\mu.
$$
Finally, it follows from Proposition~4.8 of~\cite{McM-ren2} that the domains $U_f$ and $\psi(U_f)$ contain neighborhoods $N_{2s}([-1,0])$ and $N_{s}([-1,1])$ respectively, for some $s$ that depends only on~$\mu$.
\end{proof}




Assume that a set $A\subset\tilde{\mathfrak A}^{(1,+\infty)}$ is contained and is relatively compact in $\tilde{\mathfrak A}_r^{(1,+\infty)}$, for some $r>0$. Then we let $\overline{A}^r\subset\tilde{\mathfrak A}_r^{(1,+\infty)}$ denote the closure of $A$ with respect to the $\dist_r$-metric.

\begin{lemma}\label{H_mu_compact_lemma}
For a bounded set $I\subset(1,+\infty)$ and a real number $\mu\in(0,1)$, let $s\in\bbR$ be such that $0<s\le s(\mu)$, where $s(\mu)$ is the same as in Lemma~\ref{definite_r_lemma}. Then the set $\mathbf H^I(\mu)$ is relatively compact in $\tilde{\mathfrak A}_s^{\overline I}$, and 
if a map $f\in \overline{\mathbf H^I(\mu)}^s$ has critical exponent $\alpha$, then $\alpha\in \overline I$, and $f\in{\mathbf H^{\{\alpha\}}(\mu/2)}$.
\end{lemma}
\begin{proof}
Let $\mathcal F^I(\mu)$ be the family of all pairs $(U_f, \psi)$, such that $\psi$ is a univalent analytic map of the domain $U_f$, and both $\psi$ and $U_f$ satisfy Definition~\ref{H_mu_def} for some $f\in\mathbf H^I(\mu)$.
Let $\mathcal E^I(\mu)$ be the family of all marked domains $(U_f,0)$, such that $(U_f,\psi)\in\mathcal F^I(\mu)$, for some map $\psi$.
According to Lemma~\ref{definite_r_lemma} and Theorem~5.2 from~\cite{McM-ren1}, the family $\mathcal E^I(\mu)$ is relatively compact in the space of all marked topological disks with respect to the Carath\'eodory topology. Furthermore, it follows from Definition~\ref{H_mu_def} that the sets $\psi(U_f)$ are uniformly bounded for all $(U_f,\psi)\in \mathcal F^I(\mu)$. Similarly, since the set $I$ is bounded, it follows from property~(ii) of Definition~\ref{H_mu_def} that the sets $U_f$ are uniformly bounded for all $(U_f,\psi)\in \mathcal F^I(\mu)$. 

Now, since all maps $\psi$ that appear in $\mathcal F^I(\mu)$, belong to $\tilde{\mathfrak B}_s$ and are uniformly bounded, then by Montel's theorem, every sequence in $\mathcal F^I(\mu)$ has a subsequence $(U_n,\psi_n)$, such that $\psi_n$ converge to a map $\tilde{\psi}$ which is analytic in $N_s([-1,0])$. Relative compactness of $\mathcal E^I(\mu)$ implies that after passing to a subsequence again, we ensure that the sequences of marked domains $(U_n,0)$ and $(\psi_n(U_n),\psi_n(0))$ converge to $(U,0)$ and $(V,\tilde{\psi}(0))$ respectively in Carath\'eodory topology. Finally, it follows from Theorem~5.6 of~\cite{McM-ren1} that the limit map $\tilde{\psi}$ is defined and univalent in $U\Supset N_s([-1,0])$. The latter immediately implies the lemma.
\end{proof}

The following theorem is a direct consequence of real {\it a priori} bounds (see e.g.~\cite{deMelo_vanStrien}).

\begin{theorem}[{\bf Real bounds}]
\label{real bounds_extended}
For every finite non-empty set $\Theta\subset\mathfrak P$, there exists a family of unimodal maps $\hat{\mathcal S}_\Theta^\infty\subset\mathcal S_\Theta^\infty$, such that the following holds:

(i) for every bounded set $I\subset(1,+\infty)$ and for every $\mu\in(0,1)$, $s\in\bbR$, such that $0<s\le s(\mu)$, where $s(\mu)$ is the same as in Lemma~\ref{definite_r_lemma}, the set $\overline{\mathbf H^I(\mu)}^s\cap \hat{\mathcal S}_\Theta^\infty$ is compact in $\dist_s$-metric;

(ii) for every positive real number $r>0$ and every  relatively compact family $S\subset\tilde{\mathfrak A}_r^I$ of unimodal maps, there exists $K_2=K_2(r, S)\in\bbN$ such that for every $n\ge K_2$, we have
$\cR^n(S\cap\mathcal S_\Theta^\infty)\subset\hat{\mathcal S}_\Theta^\infty$.
\end{theorem}

The next statement is a form of complex {\it a priori} bounds:

\begin{theorem}[{\bf Complex bounds}]
\label{complex bounds_extended}
For every compact set $I\subset(1,+\infty)$, there exists a constant $\mu=\mu(I)>0$  such that
the following holds. For every positive real number $r>0$ and every pre-compact family $S\subset\tilde{\mathfrak A}_r^I$ of unimodal maps, there exists $K_1=K_1(r, S)\in\bbN$ such that if $f\in S$ is
an $n+1$ times renormalizable unimodal map, where $n\geq K_1$, then 
for every $g\in\tilde{\mathfrak A}_r^I$, sufficiently close to $f$ in $\dist_r$-metric, we have $\cR^n(g)\in\mathbf H^I(\mu)$.
\end{theorem}

\noindent
We note that the standard proofs of complex {\it a priori} bounds (see \cite{Lyubich-yam} and references therein) are given for analytic unimodal maps with even critical exponents $\alpha\in 2\NN$. However, in the above-stated form, the standard proofs apply to the case of a general exponent $\alpha$ {\it mutatis mutandis}.

\subsection{Global stable sets}
In this subsection we prove properties (ii) and (i) of Theorem~\ref{Global_attractor_theorem_expanded}.

For each $k\in\bbN$, define $\mu_k=\mu([2k-1,2k+1])$, where $\mu([2k-1,2k+1])$ is the same as in Theorem~\ref{complex bounds_extended}. According to Remark~\ref{r_small_remark}, without loss of generality we may assume that 
\begin{equation}\label{r_k_smaller_s_equation}
r(k)\le s(\mu_k/2),
\end{equation}
where $s(\mu_k/2)$ is the same as in Lemma~\ref{definite_r_lemma}.

\begin{proposition}\label{second_step_conv_proposition}
Fix a positive integer $k\in\bbN$ and a finite non-empty set $\Theta\subset\mathfrak P$. 
Let $r=r(k)$ be the same as in Theorem~\ref{renormalization_2k_theorem}.
For any open set $\mathcal O\subset\mathfrak A_r^{(1,+\infty)}$, such that $\mathcal I_\Theta^{2k}\subset\mathcal O$, there exist an open interval $I=I(\mathcal O)\subset [2k-1,2k+1]$ and a positive integer $L\in\bbN$, with the property that $2k\in I$, and for every $f\in\mathbf H^{I}(\mu_k)\cap\hat{\mathcal S}_\Theta^\infty$, we have $\cR^L(f)\in\mathcal O$.
\end{proposition}
\begin{proof}
Since $\Theta$ is a finite set, it follows from~(\ref{r_k_smaller_s_equation}) and Theorem~\ref{real bounds_extended} that the set $\overline{\mathbf H^{\{2k\}}(\mu_k/2)}^r\cap\hat{\mathcal S}_\Theta^\infty$ is compact in $\tilde{\mathfrak A}_r^{[2k-1,2k+1]}$. Together with global convergence to the attractor $\mathcal I_\Theta^{2k}$, guaranteed by Theorem~\ref{renormalization_2k_theorem}, this implies existence of a positive constant $L\in\bbN$, such that 
$$
\cR^L\left(\overline{\mathbf H^{\{2k\}}(\mu_k/2)}^r\cap\hat{\mathcal S}_\Theta^\infty\right)\Subset\mathcal O, 
$$
and
$$
L>K_1(r, \mathbf H^{[2k-1,2k+1]}(\mu_k)),
$$
where $K_1(r, \mathbf H^{[2k-1,2k+1]}(\mu_k))$ is the same as in Theorem~\ref{complex bounds_extended}. (The last inequality ensures that the operator $\cR^L$ maps a neighborhood of $\overline{\mathbf H^{[2k-1,2k+1]}(\mu_k)}^r\cap\hat{\mathcal S}_\Theta^\infty$ to $\mathfrak A_r^{[2k-1,2k+1]}$.)

Now, Lemma~\ref{H_mu_compact_lemma} and continuity of the operator $\cR^L$ on the sequentially compact family $\overline{\mathbf H^{[2k-1,2k+1]}(\mu_k)}^r\cap\hat{\mathcal S}_\Theta^\infty$ imply existence of the interval $I$ that satisfies the lemma.
\end{proof}

\begin{proof}[Proof of Lemma~\ref{global_lemma}]
First, we prove property~(ii). 

Let $J_1$ be the same as in Lemma~\ref{hyperb_lemma}.
It follows from hyperbolicity of the sets $\mathcal I_\Theta^\alpha$ (c.f.~Lemma~\ref{hyperb_lemma}), that there exist an open interval $J_2\subset J_1$, such that $2k\in J_2$, and an open set $\mathcal O\subset\coprod_{\alpha\in J}\mathcal O_\alpha\subset\mathfrak A_r^{J_1}$, such that for any $\alpha\in J_2$, we have $\mathcal I_\Theta^\alpha\subset \mathcal O$, and for any unimodal map $f\in\mathcal O$ with critical exponent $\alpha\in J_2$, the sequence of iterates $\cR^N(f),\cR^{2N}(f),\cR^{3N}(f),\dots$ either eventually leaves the set $\mathcal O$, or stays in it forever and converges to the invariant set $\mathcal I_\Theta^\alpha$.

Fix $J=I(\mathcal O)\subset J_2$, where $I(\mathcal O)$ is the same as in Proposition~\ref{second_step_conv_proposition}. Now it follows from Theorem~\ref{complex bounds_extended} and Theorem~\ref{real bounds_extended} that for every $f\in \mathfrak A^J\cap\mathcal S_\Theta^\infty$, there exists a positive integer $K=K(f)\in\bbN$, such that for every $n\ge K$, we have $\cR^n(f)\in\mathbf H^{J}(\mu_k)\cap\hat{\mathcal S}_\Theta^\infty$. Together with Proposition~\ref{second_step_conv_proposition} this implies that for every $f\in \mathfrak A^J\cap\mathcal S_\Theta^\infty$, there exists a positive integer $M=M(f)\in\bbN$, such that for every $n\ge M$, we have $\cR^n(f)\in\mathcal O$. In other words, for every $f\in \mathfrak A^J\cap\mathcal S_\Theta^\infty$, the sequence $\cR(f), \cR^2(f), \cR^3(f),\dots$ eventually enters the set $\mathcal O$ and since then never leaves it. According to our choice of the set $\mathcal O$, this implies that the considered sequence of renormalizations converges to $\mathcal I_\Theta^\alpha$, where $\alpha\in J$ is the critical exponent of $f$.
Together with hyperbolicity of $\mathcal I_\Theta^\alpha$, established in Lemma~\ref{hyperb_lemma} for all $\alpha\in J$, this implies that sufficiently high renormalizations $\cR^n(f)$ belong to the stable lamination of the set $\mathcal I_\Theta^\alpha$. This means that there exists $g\in\mathcal I_\Theta^\alpha$, such that 
condition~(\ref{pair_exp_convergence_eq}) holds for all sufficiently large $m\in\bbN$ of the form $m=Nn$.


We observe that if $\rho(f)$ and $\rho(g)$ are asymptotically different, then for every $n\in\bbN$, the renormalizations $\cR^{Nn}(f)$ and $\cR^{Nn}(g)$ cannot get arbitrarily close to each other. Thus, $\rho(f)$ is asymptotically equal to $\rho(g)$. Since, according to Lemma~\ref{hyperb_lemma}, the restrictions of $\cR^N$ to $\mathcal I_\Theta^\alpha$ and $\mathcal I_\Theta^{2k}$ are topologically conjugate, Theorem~\ref{renormalization_2k_theorem} implies that the convergence~(\ref{pair_exp_convergence_eq}) holds for all sufficiently large $m\in\bbN$ of the form $m=Nn$ and for all $g\in\mathcal I_\Theta^\alpha$, such that $\rho(f)$ and $\rho(g)$ are asymptotically equal.

Now we complete the proof of property~(ii) of Theorem~\ref{Global_attractor_theorem_expanded} by showing that $\cR(\mathcal I_\Theta^\alpha)=\mathcal I_\Theta^\alpha$. 
Indeed, according to the above argument and compactness of $\mathcal I_\Theta^\alpha$, convergence of the sequence $\{\cR^{Nn}(f)\mid n\in\bbN\}$ to $\mathcal I_\Theta^\alpha$ is uniform in $f\in\cR(\mathcal I_\Theta^\alpha)$. Since $\cR^{Nn+1}(\mathcal I_\Theta^\alpha)=\cR(\mathcal I_\Theta^\alpha)$, this implies that $\cR(\mathcal I_\Theta^\alpha)=\mathcal I_\Theta^\alpha$.

Finally, we give a proof of property~(i) of Theorem~\ref{Global_attractor_theorem_expanded}. For every $\alpha\in J$, let $h_\alpha\colon\mathcal I_\Theta^{2k}\to \mathcal I_\Theta^\alpha$ be the homeomorphism from Lemma~\ref{hyperb_lemma} that conjugates the restrictions of $\cR^N$ on $\mathcal I_\Theta^{2k}$ and $\mathcal I_\Theta^\alpha$. Define the homeomorphism $\iota_\alpha\colon\mathcal I_\Theta^\alpha\to\Theta^\bbZ$ as $\iota_\alpha=\iota_{2k}\circ h_\alpha^{-1}$. Since for all $\alpha\in J$, we have $\cR(\mathcal I_\Theta^\alpha)=\mathcal I_\Theta^\alpha$, the composition $\iota_\alpha\circ\cR\circ\iota_\alpha^{-1}$ is defined for all $\alpha\in J$ and depends continuously on $\alpha$. Then, since $\Theta^\bbZ$ is a totally disconnected space, this composition must be independent from $\alpha$. Now property~(i) of Theorem~\ref{Global_attractor_theorem_expanded} follows from the fact that for $\alpha=2k$, this composition is a shift $\sigma$, which is established in Theorem~\ref{renormalization_2k_theorem}.
\end{proof}

\section{Proof of Theorem~\ref{non_even_theorem}}\label{non_even_sec}

According to the real a priori bounds (e.g., see~\cite{deMelo_vanStrien}), there exists a real constant $\mu>1$ that depends only on $\alpha$, such that for every $g\in\mathbf U^\omega\cap\mathcal S_{\mathbf P}^\infty$ with critical exponent $\alpha$, there exists a constant $K=K(g)\in\bbN$, such that for every $k\ge K$, we have
$$
\cR^{k+1}(g)=A_{p,q}\circ g_k^m\circ A_{p,q}^{-1}, \quad\text{ where}\quad g_k=\cR^k(g), 
$$ 
and $A_{p,q}$ is an affine map with $A_{p,q}'(z)\ge\mu$. Together with the local representation~(\ref{f_equals_psi_alpha_phi_equation}), this immediately implies statement~(i) of Theorem~\ref{non_even_theorem}.

Statement~(ii) of Theorem~\ref{non_even_theorem} follows from the Koebe Distortion Theorem and the real a priori bounds stated above. Indeed, for $k>K(\phi^{-1}\circ f\circ\phi)$, it follows from the real a priori bounds that the map $\mathcal F_f^k(\phi)$ is defined and univalent in $N_{r\mu^{k-K}}([-1,1])$, hence according to the Koebe Distortion Theorem, the maps $\mathcal F_f^k(\phi)$ converge to an affine map exponentially fast in $\|\cdot\|_r$-norm. Since all maps from $\mathbf{\Phi}_r$ fix the points $-1$ and $1$, the only affine map, contained in the closure of $\mathbf{\Phi}_r$, is the identity map. Thus, $\mathcal F_f^k(\phi)\to\mathrm{id}$ exponentially fast in $\|\cdot\|_r$-norm.

\bibliographystyle{amsalpha}
\bibliography{biblio}
\end{document}